\documentclass{amsart}
\usepackage{graphicx}
\usepackage{amscd}
\usepackage{amsmath}
\usepackage{amsfonts}
\usepackage{amssymb}
\newtheorem{theorem}{Theorem}
\theoremstyle{plain}

\newtheorem{corollary}{Corollary}

\newtheorem{definition}{Definition}

\newtheorem{lemma}{Lemma}

\newtheorem{proposition}{Proposition}
\newtheorem{remark}{Remark}

\numberwithin{equation}{section}

\begin{document}
\title[Self Improving Sobolev-Poincar\'{e} Inequalities]{Self Improving Sobolev-Poincar\'{e} Inequalities, Truncation and Symmetrization}
\author{Joaquim Martin$^{\ast}$}
\address{Department of Mathematics\\
Universitat Aut\`onoma de Barcelona}
\email{jmartin@mat.uab.cat}
\author{Mario Milman}
\address{Department of Mathematics\\
Florida Atlantic University}
\email{extrapol@bellsouth.net}
\urladdr{http://www.math.fau.edu/milman}
\thanks{2000 Mathematics Subject Classification Primary: 46E30, 26D10.}
\thanks{$^{\ast}$ Supported in part by MTM2007-60500 and by CURE 2005SGR00556}
\thanks{This paper is in final form and no version of it will be submitted for
publication elsewhere.}

\begin{abstract}
In \cite{MMP} we developed a new method to obtain symmetrization inequalities
of Sobolev type for functions in $W_{0}^{1,1}(\Omega)$. In this paper we
extend our method to Sobolev functions that do not vanish at the boundary.
\end{abstract}\maketitle

\section{Introduction}

In our recent paper \cite{MMP} we developed a new principle of
``symmetrization by truncation'' to obtain symmetrization inequalities of
Sobolev type via truncation. In this note we consider the corresponding
results for Sobolev spaces on domains, without assuming that the Sobolev
functions vanish at the boundary.

The explicit connection between Sobolev-Poincar\'{e} inequalities and
isoperimetric inequalities appears in the work of Maz'ya. In \cite{Ma} it is
shown that if $\Omega\subset\mathbb{R}^{n}$ is an arbitrary open set with
finite volume, $1\leq p\leq n/(n-1),$ then the Sobolev-Poincar\'{e}
\begin{equation}
\left(  \int_{\Omega}\left|  f(x)-f_{\Omega}\right|  ^{p}dx\right)  ^{1/p}\leq
C\int_{\Omega}\left|  \nabla f(x)\right|  dx,\;\forall f\in W^{1,1}(\Omega),
\label{tres}%
\end{equation}
$(f_{\Omega}=\frac{1}{\left|  \Omega\right|  }\int_{\Omega}f)$ holds if and
only if the following $p-$isoperimetric inequality is satisfied: there exists
a constant $M\in(0,\left|  \Omega\right|  )$ such that
\begin{equation}
U_{1/p}(M)=\sup\frac{\left|  \mathcal{S}\right|  ^{1/p}}{s(\partial
\mathcal{S})}<\infty, \label{tresa}%
\end{equation}
where the sup is taken over all $\mathcal{S}$ open bounded subsets of $\Omega$
such that $\Omega\cap\partial\mathcal{S}$ is a manifold of class $C^{\infty}$
and $\left|  \mathcal{S}\right|  \leq M,$ and $s\ $denotes the $(n-1)-$%
dimensional area. If (\ref{tresa}) is satisfied we shall say that $\Omega$
belongs to the Maz'ya class $\mathcal{J}_{1/p}.$ For example, if $\Omega$ is a
bounded domain, starshaped with respect to a ball, or having the cone
property, or $\Omega$ is a Lipschitz domain, then $\Omega$ belongs to the
class $\mathcal{J}_{1-1/n};$ if $\Omega$ is a $s-$John domain then $\Omega
\in\mathcal{J}_{(n-1)s/n};$ if $\Omega$ is a domain with one $\beta-$cusp then
it belong to the Mazy'a class $\mathcal{J}_{\frac{\beta(n-1)}{\beta(n-1)+1}}$
(cf. \cite{Ma}, \cite{BK}).

Sobolev-Poincar\'{e} inequalities are known to self improve. For example, if
(\ref{tres}) holds for $p=\frac{n}{n-1},$ then (cf. \cite[Theorem 2.4.1]{Zi})
the inequality
\[
\left(  \int_{\Omega}\left|  f(x)-f_{\Omega}\right|  ^{pn/(n-p)}dx\right)
^{\frac{n-p}{np}}\leq C\left(  \int_{\Omega}\left|  \nabla f(x)\right|
^{p}dx\right)  ^{1/p},
\]
holds for $1<p<n.$ More generally, if $\left|  \Omega\right|  <\infty$, and if
inequality (\ref{tres}) holds for a fixed $p,$ $1\leq p\leq n/(n-1),$ then
\begin{equation}
\left(  \int_{\Omega}\left|  f(x)-f_{\Omega}\right|  ^{s}dx\right)  ^{1/s}\leq
C\left(  \int_{\Omega}\left|  \nabla f(x)\right|  ^{q}dx\right)  ^{1/q},
\label{cuatro}%
\end{equation}
where $q<p/(p-1)$ and $s=pq/(p+q-pq).$ In particular, in some sense, ``all''
$L^{p}$ Sobolev-Poincar\'{e} inequalities follow from the Sobolev-Poincar\'{e}
inequality (\ref{tres}) or, equivalently, from a suitable version of an
isoperimetric inequality.

As is well known, the sharp versions of these $L^{p}$ inequalities fall
outside the $L^{p}$ scale and need to be formulated using $L(p,q)$ spaces.
Recently (cf. \cite{BMR}, \cite{MP}, \cite{MM}), we have shown that using a
simple modification of the definition of the $L(p,q)$ spaces we also obtain
the ``best'' results including the problematic borderline inequalities.
Moreover, these sharper limiting results cannot be obtained using, for
example, the usual extrapolations from the $L^{p}$ inequalities but require
new sharp symmetrization inequalities.

More generally, symmetrization inequalities play a fundamental role in the
study of Sobolev-Poincar\'{e} inequalities in the general setting of
rearrangement invariant spaces. In our program we formulate self improving
properties of Sobolev-Poincar\'{e} inequalities in terms of symmetrization
inequalities. In this fashion instead of showing that a particular inequality
implies other inequalities one case at a time, we aim to prove a
symmetrization inequality that implies ``all'' other Sobolev-Poincar\'{e}
inequalities. One difficulty in dealing with rearrangement inequalities on
domains is that the usual inequalities are only valid for certain range of the
values of the variable. For example, suppose that for some $1<p\leq n/(n-1),$
the Sobolev-Poincar\'{e} inequality (\ref{tres}) holds, then (cf. \cite{MM}),
\begin{equation}
f^{\ast\ast}(t)-f^{\ast}(t)\leq Ct^{1-1/p}\left|  \nabla f\right|  ^{\ast\ast
}(t),\ t\in(0,\left|  \Omega\right|  /2),\text{ }f\in W^{1,1}(\Omega
),\label{sime00}%
\end{equation}
where $f^{\ast\ast}(t)=\frac{1}{t}\int_{0}^{t}f^{\ast}(s)ds.$ However in
\cite{MM} we show that if we work with symmetrization inequalities of
``Sobolev-Poincar\'{e}'' type (i.e. inequalities where $f$ is replaced by
$f-f_{\Omega}$) then we can eliminate the restriction $t\in(0,\left|
\Omega\right|  /2)$ in (\ref{sime00})$.$ Indeed, under the assumption that
(\ref{tres}) holds for some $1<p\leq n/(n-1),$ we showed in \cite{MM} that,
for all $f\in W^{1,1}(\Omega),$ we have
\begin{equation}
\inf_{c\in\mathbb{R}}\left(  \left(  f-c\right)  ^{\ast\ast}(t)-\left(
f-c\right)  ^{\ast}(t)\right)  \leq C_{\Omega}t^{1-1/p}\left|  \nabla
f\right|  ^{\ast\ast}(t),\text{\ a.e. }t\in(0,\left|  \Omega\right|
).\label{sime01}%
\end{equation}

Notice that (\ref{sime01}) implies that for any r.i. space $X(0,\left|
\Omega\right|  ),$ with upper Boyd\footnote{The restriction on the Boyd
indices is only required to guarantee that the inequality $\left\|
g^{\ast\ast}\right\|  _{X}\leq c_{X}\left\|  g\right\|  _{X},$ holds for all
$g\in X.$} index $\beta_{X}<1$, we have (cf. \cite{MM})
\[
\inf_{c\in\mathbb{R}}\left\|  t^{1/p-1}\left(  f-c\right)  ^{\ast\ast
}(t)-\left(  f-c\right)  ^{\ast}(t)\right\|  _{X}\leq C\left\|  \nabla
f\right\|  _{X},
\]
where $C=C(n,\left|  \Omega\right|  ,X).$ For example, if $X=L^{q}$, $q>1$,
$q<\frac{p}{p-1},$ and $s=pq/(p+q-pq),$ then
\[
\left\|  f-f_{\Omega}\right\|  _{L^{s,q}(\Omega)}\leq C\left\|  \nabla
f\right\|  _{L^{q}(\Omega)},\text{ \ \ }\forall f\in W^{1,q}(\Omega).
\]
Since $L^{s,q}(\Omega)\subset L^{s}(\Omega),$ for $s>q,$ this last inequality
is the well known (optimal) improvement of (\ref{cuatro}). Moreover, in the
limiting case $q=\frac{p}{p-1},$ then $s=\infty$ and we obtain
\begin{equation}
\inf_{c\in\mathbb{R}}\left\|  f-c\right\|  _{L^{\infty,q}(\Omega)}\leq
C\left\|  \nabla f\right\|  _{L^{q}(\Omega)},\label{agregada}%
\end{equation}
where
\[
L^{\infty,q}(\Omega)=\left\{  f:\left\|  f\right\|  _{L^{\infty,q}(\Omega
)}^{q}=\int_{0}^{\left|  \Omega\right|  }\left(  f^{\ast\ast}(t)-f^{\ast
}(t)\right)  ^{q}\frac{dt}{t}<\infty\right\}  .
\]
Once again since $L(\infty,q)(\Omega)\subset BW^{q}(\Omega)\subset
e^{L^{q^{\prime}}}(\Omega)$ (see \cite{BMR}) we see that (\ref{agregada}) is a
sharpening of the classical limiting inequalities of
Brezis-Wainger-Hansson-Maz'ya-Trudinger. It follows that if we redefine the
$L(p,q)$ spaces, $1\leq p\leq\infty,$ $1\leq q\leq\infty,$ using
\[
\left\|  f\right\|  _{L^{p,q}(\Omega)}^{q}=\int_{0}^{\left|  \Omega\right|
}\left(  f^{\ast\ast}(t)-f^{\ast}(t)\right)  ^{q}t^{q/p}\frac{dt}{t},
\]
then we have an attractive unified way to formulate the sharp form of the
Sobolev-Poincar\'{e} inequalities, namely
\begin{equation}
\inf_{c\in\mathbb{R}}\left\|  f-c\right\|  _{L^{s,q}(\Omega)}\leq C\left\|
\nabla f\right\|  _{L^{q}(\Omega)},1<q\leq
p/(p-1),s=pq/(q+p-pq).\label{agregada1}%
\end{equation}
One possible objection to (\ref{agregada1}) is that the important case $q=1$
is excluded$.$ The cause for this imperfection is the presence of the ``double
star'' operation on right hand side of (\ref{sime01})$.$ On the other hand,
(\ref{sime01}), for $q=1,$ readily implies
\begin{equation}
\left\|  f-f_{\Omega}\right\|  _{L^{p,\infty}(\Omega)}\leq C\left\|  \nabla
f\right\|  _{L^{1}(\Omega)},\label{agregada2}%
\end{equation}
and therefore, by the truncation principle of Maz'ya (cf. \cite{Ha}), we can
see that (\ref{agregada2}) self-improves to (\ref{tres}) and even to the
sharper form of the Gagliardo-Nirenberg inequality (cf. \cite{KO}),
\[
\left\|  f-f_{\Omega}\right\|  _{L^{p,1}(\Omega)}\leq C\left\|  \nabla
f\right\|  _{L^{1}(\Omega)}.
\]

The ad-hoc argument that we needed to cope with the limiting case suggested to
us that one should be able to find a sharpening of the symmetrization
inequality (\ref{sime01}) that would imply ``all'' the Sobolev-Poincar\'{e}
inequalities directly. In the case of functions vanishing at the boundary of
$\Omega$ we have shown that this is indeed the case in \cite{MMP}. One of the
objectives of this paper is to formulate the correspoding inequalities without
assuming that the Sobolev functions vanish at the boundary. Our first result
is the following

\begin{theorem}
\label{teoA}Let $\Omega$ be a domain of finite measure (for simplicity we
assume from now on that $\left|  \Omega\right|  =1),$ and let $1\leq p\leq
n/(n-1).$ Then the following statements are equivalent

(i)
\begin{equation}
\left(  \int_{\Omega}\left|  f(x)-f_{\Omega}\right|  ^{p}dx\right)
^{1/p}\preceq\int_{\Omega}\left|  \nabla f(x)\right|  dx,\text{ }\forall f\in
W^{1,1}(\Omega). \label{teoA00}%
\end{equation}
(ii) For each $f\in W^{1,1}(\Omega)$ there exists $r_{f}\in\mathbb{R}$ such
that
\begin{equation}
s^{\frac{1}{p}-1}[\left(  f-r_{f}\right)  ^{\ast\ast}(s)-\left(
f-r_{f}\right)  ^{\ast}(s)]\preceq\int_{0}^{t}\left|  \nabla f\right|  ^{\ast
}(s)ds, \label{teoA20}%
\end{equation}
and
\begin{equation}
\int_{0}^{t}s^{\frac{1}{p}-1}[\left(  f-r_{f}\right)  ^{\ast\ast}(s)-\left(
f-r_{f}\right)  ^{\ast}(s)]ds\preceq\int_{0}^{t}\left|  \nabla f\right|
^{\ast}(s)ds. \label{teoA2}%
\end{equation}
(iii) For any r.i. space\footnote{For a rearrangement invariant space (r.i.
space) $X(\Omega)$ we let $\hat{X}=\hat{X}(0,1)$ be its representation as a
function space on $(0,1)$ (if $X(\Omega)=L^{p}(\Omega)$ we shall write $L^{p}$
instead of $\hat{L}^{p}).$ We refer to \cite{BS} for further information about
r.i. spaces.} $X(\Omega)$ and for each $f\in W_{X}^{1}(\Omega)=\{f\in
X(\Omega):\nabla f\in X(\Omega)\},$ we have
\begin{equation}
\inf_{c\in\mathbb{R}}\left\|  s^{\frac{1}{p}-1}[\left(  f-c\right)  ^{\ast
\ast}(s)-\left(  f-c\right)  ^{\ast}(s)\right\|  _{\hat{X}}\preceq\left\|
\nabla f\right\|  _{X(\Omega)}. \label{final}%
\end{equation}
(iv)
\[
\left\|  f-f_{\Omega}\right\|  _{L^{p,1}(\Omega)}\preceq\left\|  \nabla
f\right\|  _{L^{1}(\Omega)},\text{ }\forall f\in W^{1,1}(\Omega).
\]
\end{theorem}

As usual, the symbol $f\simeq g$ will indicate the existence of a universal
constant $C>0$ (independent of all parameters involved) so that $(1/C)f\leq
g\leq C\,f$, while the symbol $f\preceq g$ means that for a suitable constant
$C,$ $f\leq C\,g,$ and likewise $f\succeq g$ means that $f\geq Cg.$

We note that Theorem \ref{teoA} improves on Theorem 1 of \cite{MMP} in three
respects: (i) we do not assume that the Sobolev functions vanish at the
boundary, (ii) in (\ref{final}) we have eliminated the restriction on the Boyd
index of $X$ we had in \cite{MMP} (this is due to our use of Lemma \ref{01}
below), and finally (iii) in \cite{MMP} we only considered the limiting case
$p=\frac{n}{n-1}.$

In our second main result we show that for $p=\frac{n}{n-1},$ Theorem
\ref{teoA} is sharp in the setting of r.i. spaces, and moreover that the
verification of Sobolev-Poincar\'{e} inequalities is reduced to establish the
boundedness of a certain one-dimensional Hardy type operator acting on
functions defined on $(0,1).$ Interestingly this reduction is not possible for
$p\neq\frac{n}{n-1}$ (see Proposition \ref{Har} below).

\begin{theorem}
\label{teoB}Let $\Omega$ be a domain with $\left|  \Omega\right|  =1,$ and let
$X(\Omega),$ $Y(\Omega)$ be two r.i. spaces$.$ Assume that the following
Sobolev-Poincar\'{e} inequality holds
\begin{equation}
\left(  \int_{\Omega}\left|  f(x)-f_{\Omega}\right|  ^{n/(n-1)}dx\right)
^{\frac{n-1}{n}}\preceq\int_{\Omega}\left|  \nabla f(x)\right|  dx,\text{
}\forall f\in W^{1,1}(\Omega). \label{poi}%
\end{equation}
Then the following statements are equivalent

(i)
\[
\left\|  f\right\|  _{\hat{Y}}\preceq\left\|  s^{-1/n}[f^{\ast\ast}%
(s)-f^{\ast}(s)]\right\|  _{\hat{X}}+\left\|  f\right\|  _{L^{1}}.
\]
(ii)
\[
\left\|  \int_{t}^{1}s^{1/n}f(s)\frac{ds}{s}\right\|  _{\hat{Y}}%
\preceq\left\|  f\right\|  _{\hat{X}},\ \forall f\in\hat{X},\text{ }f\geq0.
\]
(iii)
\[
\left\|  f-f_{\Omega}\right\|  _{Y(\Omega)}\preceq\left\|  \nabla f\right\|
_{X(\Omega)}.
\]
\end{theorem}

Finally, we also consider suitable variants of the Polya-Sz\"{e}go
symmetrization principle in a formulation that does not require the functions
to vanish at the boundary

\begin{theorem}
\label{teo3}(cf. Theorem \ref{Tpolya} below) \ Let $\Omega\in$ $\mathcal{J}%
_{1-1/n}$, and let $X(\Omega)$ be a r.i. space. Then
\[
\inf_{c\in\mathbb{R}}\left\|  \nabla(f-c)^{\circ}\right\|  _{\tilde{X}%
(B)}\preceq\left\|  \nabla f\right\|  _{X(\Omega)},\text{ for all }f\in
W^{1,1}(\Omega),
\]
where $f^{\circ}$ is the symmetric spherical decreasing rearrangement of
$f$\textbf{\ }and\textbf{ }$\tilde{X}(B)$ is the version of $X(\Omega)$ on a
ball $B$ centered at zero with measure 1 (see Section \ref{sec4} below).
\end{theorem}

Using Theorem \ref{teo3}, and the characterization of the $X-$modulus of
continuity as a $K-$functional (cf. \cite{hs}), it follows as in \cite{MM1} that

\begin{theorem}
\label{teo4}Let $\Omega$ be an open domain in $\mathbb{R}^{n}$ with Lipschitz
boundary with $\left|  \Omega\right|  =1,$ and let $X\left(  {\Omega}\right)
$ be a r.i. space. Then for all $f\in X\left(  {\Omega}\right)  ,$%
\[
\inf_{c\in\mathbb{R}}\omega_{\tilde{X}(B)}(\left(  f-c\right)  ^{\circ
},t)\preceq\omega_{X(\Omega)}(f,t),
\]
where $\omega_{X(\Omega)}(f,t)$ is the $X$-modulus of continuity of $f$ (see
(\ref{modulo}) below).
\end{theorem}

The paper is organized as follows: in Section \ref{sec1} we deal with the
modifications necessary to make the ``symmetrization by \ truncation
principle'' method of \cite{MMP} available in our setting, in particular this
section contains a proof that (\ref{teoA00}) implies theorem \ref{teoA} (ii),
we then complete the proofs of Theorems \ref{teoA} and \ref{teoB} in Section
\ref{sec2} while we prove Theorems \ref{teo3} and \ref{teo4} in section
\ref{sec4}.

\section{\label{sec1}Rearrangement Inequalities on Domains by Truncation}

Let $\Omega\subset\mathbb{R}^{n}$ be a domain which, for simplicity, we
suppose is such that $\left|  \Omega\right|  =1.$ In this section we prove
(cf. Theorem \ref{perplejo} below) that (\ref{teoA00}) implies by
symmetrization by truncation the rearrangement inequalities (\ref{teoA20}) and
(\ref{teoA2}) of Theorem \ref{teoA}. These results are variants of
symmetrization inequalities, which for functions vanishing at the boundary,
have appeared in articles by Bastero-Milman-Ruiz \cite{BMR}, Martin-Milman
\cite{MM}, Mazy'a \cite{Ma}, Talenti \cite{Ta}, Martin-Milman-Pustylnik
\cite{MMP}, etc. Our method of proof is by ``symmetrization by truncation''
developed recently in \cite{MMP}, therefore we shall only indicate briefly the
necessary changes and refer the reader to \cite{MMP} for complete details.

Throughout this section we shall assume that the following
Sobolev-Poincar\'{e} inequality holds
\begin{equation}
\left(  \int_{\Omega}\left|  f(x)-f_{\Omega}\right|  ^{p}dx\right)
^{1/p}\preceq\int_{\Omega}\left|  \nabla f(x)\right|  dx,\text{ for all
\ }f\in W^{1,1}(\Omega). \label{dosdos}%
\end{equation}

We now formally introduce the truncations we use

\begin{definition}
Let $f$ be a positive measurable function. Let $0<t_{1}<t_{2}<\infty.$ The
truncations $f_{t_{1}}^{t_{2}}$ of $f$ are defined by
\[
f_{t_{1}}^{t_{2}}(x)=\left\{
\begin{array}
[c]{ll}%
t_{2}-t_{1} & \text{if }f(x)>t_{2},\\
f(x)-t_{1} & \text{if }t_{1}<f(x)\leq t_{2},\\
0 & \text{if }f(x)\leq t_{1}.
\end{array}
\right.
\]
\end{definition}

The next useful result is a simple elementary fact that we state without proof.

\begin{lemma}
\label{ttt}Let $(X,\mu)$ be a finite measure space. If $w\geq0$ is a
measurable function such that $\mu(\left\{  w=0\right\}  )\geq\mu(X)/2$, then
for every $t>0$%
\[
\mu(\left\{  x\in X:w(x)\geq t\right\}  )\leq2\inf_{c\in\mathbb{R}}%
\mu(\left\{  x\in X:\left|  w(x)-c\right|  \geq t/2\right\}  ).
\]
\end{lemma}

We now state and prove the main result of this section (cf. also \cite{MMP})

\begin{theorem}
\label{perplejo}Let $f\in W^{1,1}(\Omega)$ then there exists $r_{f}%
\in\mathbb{R}$ such that

a.
\[
\int_{0}^{t}s^{1/p}\left(  -(f-r_{f})^{\ast}\right)  ^{^{\prime}}%
(s))ds\preceq\int_{0}^{t}\left|  \nabla f\right|  ^{\ast}(s)ds.
\]

b.
\[
s^{\frac{1}{p}-1}[\left(  f-r_{f}\right)  ^{\ast\ast}(s)-\left(
f-r_{f}\right)  ^{\ast}(s)]\preceq\int_{0}^{t}\left|  \nabla f\right|  ^{\ast
}(s)ds.
\]

c.
\[
\int_{0}^{t}s^{\frac{1}{p}-1}[\left(  f-r_{f}\right)  ^{\ast\ast}(s)-\left(
f-r_{f}\right)  ^{\ast}(s)]\frac{ds}{s}\preceq\int_{0}^{t}\left|  \nabla
f\right|  ^{\ast}(s)ds.
\]
\end{theorem}

\begin{proof}
Let $r_{f}$ be such that
\[
\left|  \left\{  f\geq r_{f}\right\}  \right|  \geq1/2\text{ \ and \ }\left|
\left\{  f\leq r_{f}\right\}  \right|  \geq1/2.
\]
Let $u=\left(  f-r_{f}\right)  \chi_{\left\{  f\geq r_{f}\right\}  }$ and
$v=\left(  r_{f}-f\right)  \chi_{\left\{  f\leq r_{f}\right\}  }.$ Consider
the truncations $u_{t_{1}}^{t_{2}}$ of $u.$ Then,
\[
\left|  \left\{  u_{t_{1}}^{t_{2}}=0\right\}  \right|  \geq1/2.
\]
Thus, by Lemma \ref{ttt} and inequality (\ref{dosdos}), we see that for all
$t>0,$%
\begin{align*}
\left|  \left\{  u_{t_{1}}^{t_{2}}\geq t\right\}  \right|  ^{1/p}t  &
\leq2^{1/p+1}\inf_{c\in\mathbb{R}}\left|  \left\{  \left|  u_{t_{1}}^{t_{2}%
}-c\right|  \geq t/2\right\}  \right|  ^{1/p}t/2\\
&  \preceq\inf_{c\in\mathbb{R}}\left(  \int_{\Omega}\left|  u_{t_{1}}^{t_{2}%
}-c\right|  ^{p}dx\right)  ^{1/p}\\
&  \preceq\int_{\left\{  t_{1}<u\leq t_{2}\right\}  }\left|  \nabla
f(x)\right|  dx.
\end{align*}
Let $t=t_{2}-t_{1},$ then
\[
(t_{2}-t_{1})\left|  \left\{  u_{t_{1}}^{t_{2}}\geq t_{2}-t_{1}\right\}
\right|  ^{1/p}\preceq\int_{\left\{  t_{1}<u\leq t_{2}\right\}  }\left|
\nabla f(x)\right|  dx.
\]
The last inequality combined with
\[
\left|  \left\{  u\geq t_{2}\right\}  \right|  =\left|  \left\{  u_{t_{1}%
}^{t_{2}}\geq t_{2}-t_{1}\right\}  \right|
\]
yields
\[
(t_{2}-t_{1})\left|  \left\{  u\geq t_{2}\right\}  \right|  ^{1/p}\preceq
\int_{\left\{  t_{1}<u\leq t_{2}\right\}  }\left|  \nabla f(x)\right|  dx.
\]
Similarly,
\[
(t_{2}-t_{1})\left|  \left\{  v\geq t_{2}\right\}  \right|  ^{1/p}\preceq
\int_{\left\{  t_{1}<v\leq t_{2}\right\}  }\left|  \nabla f(x)\right|  dx.
\]
Note that $\left|  f-r_{f}\right|  =u+v,$ then from the definition of $u$ and
$v$, it is plain that for \ $0<\alpha<\beta,$
\[
\left\{  \beta>u+v\geq\alpha\right\}  =\left\{  \beta>u\geq\alpha\right\}
\cup\left\{  \beta>v\geq\alpha\right\}  .
\]
Thus,
\begin{align*}
(t_{2}-t_{1})\left|  \left\{  \left|  f-r_{f}\right|  \geq t_{2}\right\}
\right|  ^{1/p}  &  =(t_{2}-t_{1})\left(  \left|  \left\{  u\geq
t_{2}\right\}  \right|  +\left|  \left\{  v\geq t_{2}\right\}  \right|
\right)  ^{1/p}\\
&  \leq(t_{2}-t_{1})\left(  \left|  \left\{  u\geq t_{2}\right\}  \right|
^{1/p}+\left|  \left\{  v\geq t_{2}\right\}  \right|  ^{1/p}\right) \\
&  \preceq\left(  \int_{\left\{  t_{1}<u\leq t_{2}\right\}  }\left|  \nabla
f(x)\right|  dx+\int_{\left\{  t_{1}<v\leq t_{2}\right\}  }\left|  \nabla
f(x)\right|  dx\right) \\
&  =\int_{\left\{  t_{1}<\left|  f-r_{f}\right|  \leq t_{2}\right\}  }\left|
\nabla f(x)\right|  dx.
\end{align*}
Apply the previous inequality using $t_{1}=\left(  f-r_{f}\right)  ^{\ast
}(s+h)$ and $t_{2}=\left(  f-r_{f}\right)  ^{\ast}(s),$ where $s,h>0.$ Then
dividing the resulting inequality by $h$ and letting $h\rightarrow0$
(following the corresponding argument in \cite{Ta} and \cite{MMP}) we arrive
at
\begin{equation}
s^{1/p}\left(  -(f-r_{f})^{\ast}\right)  ^{^{\prime}}(s)\preceq\frac{\partial
}{\partial s}\int_{\left\{  \left|  f-r_{f}\right|  >(f-r_{f})^{\ast
}(s)\right\}  }\left|  \nabla f(x)\right|  dx. \label{pol01}%
\end{equation}
Therefore
\begin{equation}
\int_{0}^{t}s^{1/p}\left(  -(f-r_{f})^{\ast}\right)  ^{^{\prime}}%
(s))ds\preceq\int_{0}^{t}\left|  \nabla f\right|  ^{\ast}(s)ds, \label{pol02}%
\end{equation}
follows. To prove (b) we use the definitions and integration by parts to get
\begin{align*}
\left(  f-r_{f}\right)  ^{\ast\ast}(t)-\left(  f-r_{f}\right)  ^{\ast}(t)  &
=\frac{1}{t}\int_{0}^{t}\left(  \left(  f-r_{f}\right)  ^{\ast}(s)-\left(
f-r_{f}\right)  ^{\ast}(t\right)  )ds\\
&  =\frac{1}{t}\int_{0}^{t}s^{1-1/p}s^{1/p}\left(  -(f-r_{f})^{\ast}\right)
^{^{\prime}}(s)ds\\
&  \leq\frac{t^{1-1/p}}{t}\int_{0}^{t}s^{1/p}\left(  -(f-r_{f})^{\ast}\right)
^{^{\prime}}(s)ds,
\end{align*}
and we conclude by (\ref{pol02}).

For the proof of (c) we integrate
\[
s^{1/p-1}[\left(  f-r_{f}\right)  ^{\ast\ast}(s)-\left(  f-r_{f}\right)
^{\ast}(s)]=s^{1/p-2}\int_{0}^{s}u\left(  -(f-r_{f})^{\ast}\right)
^{^{\prime}}(u)du
\]
and integrate by parts (cf. \cite{MMP}).
\end{proof}

\section{\label{sec2}Proof of Theorems \ref{teoA} and \ref{teoB}}

In order to avoid putting conditions on the indices of the r.i. spaces we
shall need the following technical result, which is implicit in \cite{Ci}, and
whose proof we provide in an appendix.

\begin{lemma}
\label{01}Let $g,h$ be two positive measurable functions on $(0,\infty)$ such
that
\begin{equation}
g(s)\preceq h^{\ast\ast}(s),\text{for all }s\in(0,\infty), \label{h1}%
\end{equation}
and
\begin{equation}
\int_{0}^{t}g(s)\preceq\int_{0}^{t}h^{\ast}(s)ds,\text{ for all }s\in
(0,\infty). \label{h2}%
\end{equation}

Then
\[
\int_{0}^{t}g^{\ast}(s)ds\preceq\int_{0}^{t}h^{\ast}(s)ds,\text{ for all }%
t\in(0,\infty),
\]
and therefore for any r.i. space $X$%
\[
\left\|  g\right\|  _{X}\preceq\left\|  h\right\|  _{X}.
\]
\end{lemma}

\subsection{The proof of Theorem \ref{teoA}}

\begin{proof}
In Section \ref{sec1} we proved that $(i)\rightarrow(ii).$

$(ii)\rightarrow(iii)$. Applying Lemma \ref{01} with
\[
g(s)=s^{1/p-1}\left(  f-r_{f}\right)  ^{\ast\ast}(s)-\left(  f-r_{f}\right)
^{\ast}(s)\ \text{and }h^{\ast}(s)=\left|  \nabla f\right|  ^{\ast}(s)
\]
we get
\[
\inf_{c\in\mathbb{R}}\left\|  s^{\frac{1}{p}-1}[\left(  f-c\right)  ^{\ast
\ast}(s)-\left(  f-c\right)  ^{\ast}(s)\right\|  _{\hat{X}}\preceq\left\|
\nabla f\right\|  _{X(\Omega)}.
\]

$(iii)\rightarrow(iv).$ Applying (iii) with $X(\Omega)=L^{1}(\Omega)$ we get
\[
\inf_{c\in\mathbb{R}}\int_{0}^{1}s^{1/p-1}\left[  \left(  f-c\right)
^{\ast\ast}(s)-\left(  f-c\right)  ^{\ast}(s)\right]  ds\preceq\left\|  \nabla
f\right\|  _{L^{1}(\Omega)}.
\]
We then note that
\[
\int_{0}^{1}s^{1/p-1}\left[  \left(  f-c\right)  ^{\ast\ast}(s)-\left(
f-c\right)  ^{\ast}(s)\right]  ds\simeq\left\|  f-c\right\|  _{L^{p,1}%
({\Omega})},
\]
and conclude with
\[
\left\|  f-f_{\Omega}\right\|  _{L^{p,1}(\Omega)}\leq2\inf_{c\in\mathbb{R}%
}\left\|  f-c\right\|  _{L^{p,1}({\Omega})}\preceq\left\|  \left|  \nabla
f\right|  \right\|  _{L^{1}(\Omega)}.
\]

$(iv)\rightarrow(i)$ This implication is trivial since
\[
L^{p,1}({\Omega})\subset L^{p}({\Omega}).
\]
\end{proof}

\subsection{The proof of Theorem \ref{teoB}}

\begin{proof}
$(i)\rightarrow(ii).$ Given $f\in\hat{X},$ let $h(t)=\int_{t}^{1}%
s^{1/n}\left|  f(s)\right|  \frac{ds}{s}$, then $h(t)=h^{\ast}(t).$
Consequently, by Fubini,
\begin{align*}
h^{\ast\ast}(t)-h^{\ast}(t)  &  =\frac{1}{t}\int_{0}^{t}h(s)ds=\frac{1}{t}%
\int_{0}^{t}\left(  \int_{x}^{1}s^{1/n}\left|  f(s)\right|  \frac{ds}%
{s}\right)  dx-h(t)\\
&  =\frac{1}{t}\int_{0}^{t}s^{1/n}\left|  f(s)\right|  ds.
\end{align*}
Also note that
\[
\left\|  h\right\|  _{L^{1}}\leq\int_{0}^{1}s^{1/n}\left|  f(s)\right|
ds\leq\left\|  f\right\|  _{L^{1}}.
\]
Consequently by ($i$)
\begin{align*}
\left\|  \int_{t}^{1}s^{1/n}\left|  f(s)\right|  \frac{ds}{s}\right\|
_{\hat{Y}}  &  =\left\|  h\right\|  _{\hat{Y}}\\
&  \preceq\left\|  t^{-1/n}\left(  h^{\ast\ast}(t)-h^{\ast}(t)\right)
\right\|  _{\hat{X}}+\left\|  h\right\|  _{L^{1}}\\
&  =\left\|  t^{-1/n}\frac{1}{t}\int_{0}^{t}s^{1/n}\left|  f(s)\right|
ds\right\|  _{\hat{X}}+\left\|  f\right\|  _{L^{1}}\\
&  \preceq\left\|  f\right\|  _{\hat{X}},
\end{align*}
where in the last inequality we used the fact that $\left\|  f\right\|
_{L^{1}}\leq\left\|  f\right\|  _{\hat{X}},$ and that for any $\alpha
>0,\left\|  t^{-\alpha}\frac{1}{t}\int_{0}^{t}\left|  g(s)\right|  ds\right\|
_{\hat{X}}\leq\frac{1}{\alpha}\left\|  g\right\|  _{\hat{X}}$ (see \cite[Lemma
2.7]{MP}).

$(ii)\rightarrow(iii).$ Pick $r_{f}\in\mathbb{R},$ such that
\begin{equation}
\left\|  s^{-1/n}[\left(  f-r_{f}\right)  ^{\ast\ast}(s)-\left(
f-r_{f}\right)  ^{\ast}(s)\right\|  _{\hat{X}}\leq2\inf_{c\in\mathbb{R}%
}\left\|  s^{-1/n}[\left(  f-c\right)  ^{\ast\ast}(s)-\left(  f-c\right)
^{\ast}(s)\right\|  _{\hat{X}} \label{nueva}%
\end{equation}
By the fundamental theorem of calculus
\[
\left(  f-r_{f}\right)  ^{\ast\ast}(t)=\int_{t}^{1}[\left(  f-r_{f}\right)
^{\ast}(s)-\left(  f-r_{f}\right)  ^{\ast}(s)]\frac{ds}{s}+\int_{0}^{1}\left(
f-r_{f}\right)  ^{\ast}(s)ds.
\]
Thus,
\begin{align*}
\left\|  f-r_{f}\right\|  _{Y}  &  \leq\left\|  \left(  f-r_{f}\right)
^{\ast\ast}\right\|  _{\hat{Y}}\\
&  =\left\|  \int_{t}^{1}s^{1/n}\left(  s^{-1/n}[\left(  f-r_{f}\right)
^{\ast}(s)-\left(  f-r_{f}\right)  ^{\ast}(s)]\right)  \frac{ds}{s}\right\|
_{\hat{Y}}\\
&  +\left\|  f-r_{f}\right\|  _{L^{1}(\Omega)}.\\
&  \preceq\left\|  s^{-1/n}[\left(  f-r_{f}\right)  ^{\ast}(s)-\left(
f-r_{f}\right)  ^{\ast}(s)]\right\|  _{\hat{X}}+\left\|  f-r_{f}\right\|
_{L^{1}(\Omega)}\\
&  \preceq\left\|  \nabla f\right\|  _{X(\Omega)}+\left\|  f-r_{f}\right\|
_{L^{1}(\Omega)}\text{ \ \ \ \ \ \ \ (by (\ref{nueva}) and (\ref{final})).}%
\end{align*}
Therefore
\begin{equation}
\inf_{c\in\mathbb{R}}\left\|  f-c\right\|  _{Y(\Omega)}\preceq\left\|  \nabla
f\right\|  _{X(\Omega)}+\inf_{c\in\mathbb{R}}\left\|  f-c\right\|
_{L^{1}(\Omega)}. \label{sabia}%
\end{equation}
To estimate the second term to the right we observe that
\begin{align*}
\inf_{c\in\mathbb{R}}\left\|  f-c\right\|  _{L^{1}(\Omega)}  &  \leq\inf
_{c\in\mathbb{R}}\left(  \int_{\Omega}\left|  f(x)-c\right|  ^{n/(n-1)}%
dx\right)  ^{(n-1)/n}\\
&  \preceq\int_{\Omega}\left|  \nabla f(x)\right|  dx\text{ \ \ \ (by
(\ref{poi}))}\\
&  \leq\left\|  \nabla f\right\|  _{X(\Omega)}.
\end{align*}
Inserting this estimate back into (\ref{sabia}) we find that
\[
\inf_{c\in\mathbb{R}}\left\|  f-c\right\|  _{Y(\Omega)}\preceq\left\|  \nabla
f\right\|  _{X(\Omega)},
\]
which combined with the elementary inequality
\[
\left\|  f-f_{\Omega}\right\|  _{Y(\Omega)}\preceq2\inf_{c\in\mathbb{R}%
}\left\|  f-c\right\|  _{Y(\Omega)},
\]
gives us $\left(  iii\right)  $.

$(iii)\rightarrow(ii).$ We assume, without loss of generality, that
$0\in\Omega$. Let $\sigma>0$ so that the ball centered at $0$ and having
measure $\sigma$ is contained in $\Omega.$ Given a positive function $g\in
\hat{X},$ with supp $g\subset\lbrack0,\sigma]$, define
\[
u(x)=\int_{\gamma_{n}\left|  x\right|  ^{n}}^{1}g(s)s^{1/n-1}ds\text{, \ }%
\]
where $\gamma_{n}=$ measure of the unit ball in $\mathbb{R}^{n}$. Observe that
for $h\in\hat{X}$ we have that
\[
\left|  \left\{  x\in B:h(\gamma_{n}\left|  x\right|  ^{n})>\lambda\right\}
\right|  =\left|  \left\{  t\in(0,1):h(t)>\lambda\right\}  \right|  .
\]
Consequently
\[
u^{\ast}(t)=\int_{t}^{1}s^{1/n}g(s)\frac{ds}{s}.
\]
Moreover, and easy computation shows that $\left|  \nabla u\right|
(x)=ng(\gamma_{n}\left|  x\right|  ^{n}).$ It follows from $(iii)$ that
\begin{align*}
\left\|  u\right\|  _{Y(\Omega)}  &  =\left\|  u^{\ast}\right\|  _{\hat{Y}%
}=\left\|  \int_{t}^{1}g(s)s^{1/n-1}ds\text{ }\right\|  _{\hat{Y}}%
\preceq\left\|  \nabla u \right\|  _{X(\Omega)}+\left\|  u_{\Omega}\right\|
_{Y(\Omega)}\\
&  =\left\|  g\right\|  _{\hat{X}}+\left\|  u_{\Omega}\right\|  _{Y(\Omega)}%
\end{align*}

We conclude observing that
\begin{align*}
\left\|  u_{\Omega}\right\|  _{Y(\Omega)}  &  =\left\|  \int_{\Omega}%
\int_{\gamma_{n}\left|  x\right|  ^{n}}^{1}g(s)s^{1/n-1}ds\right\|
_{Y(\Omega)}\leq\left\|  \int_{\gamma_{n}\left|  x\right|  ^{n}}^{1}\left|
g(s)\right|  s^{1/n-1}ds\right\|  _{L^{1}(\Omega)}\\
&  =\left\|  \int_{t}^{1}s^{1/n}\left|  g(s)\right|  \frac{ds}{s}\right\|
_{L^{1}}\leq\left\|  \int_{t}^{1}\left|  g(s)\right|  \frac{ds}{s}\right\|
_{L^{1}}\leq\left\|  g\right\|  _{L^{1}}\leq\left\|  g\right\|  _{\hat{X}}.
\end{align*}

Now, let $g\geq0$ be an arbitrary function from $\hat{X}.$ Then
\begin{align*}
\left\|  \int_{t}^{1}g(s)s^{1/n-1}ds\text{ }\right\|  _{\hat{Y}}  &
\leq\left\|  \int_{t}^{1}g(s)s^{1/n-1}\chi_{(0,\sigma)}(s)ds\text{ }\right\|
_{\hat{Y}}\\
&  +\left\|  \int_{t}^{1}g(s)s^{1/n-1}\chi_{(\sigma,1)}(s)ds\text{ }\right\|
_{\hat{Y}}\\
&  \leq\left\|  g\right\|  _{\hat{X}}+\left\|  \int_{t}^{1}g(s)s^{1/n-1}%
\chi_{(\sigma,1)}(s)ds\text{ }\right\|  _{\hat{Y}}.
\end{align*}
The last term on the right hand side can be readily estimated using
Minkowski's inequality
\begin{align*}
\left\|  \int_{t}^{1}g(s)s^{1/n-1}\chi_{(\sigma,1)}(s)ds\text{ }\right\|
_{\hat{Y}}  &  \leq\left\|  g\right\|  _{\hat{Y}}\int_{\sigma}^{1}%
s^{1/n-1}ds\\
&  \preceq\left\|  g\right\|  _{\hat{Y}}%
\end{align*}
and $(ii)$ follows$.$

$(ii)\rightarrow(i).$ By the fundamental theorem of calculus
\[
f^{\ast\ast}(t)=\int_{t}^{1}[f^{\ast\ast}(s)-f^{\ast}(s)]\frac{ds}{s}+\int
_{0}^{1}f^{\ast}(s)ds.
\]
Thus
\begin{align*}
\left\|  f\right\|  _{\hat{Y}}  &  \leq\left\|  \int_{t}^{1}s^{1/n}%
s^{-1/n}[f^{\ast\ast}(s)-f^{\ast}(s)]\frac{ds}{s}\right\|  _{\hat{Y}}+\left\|
f\right\|  _{L^{1}}\\
&  \preceq\left\|  s^{-1/n}[f^{\ast\ast}(s)-f^{\ast}(s)]\right\|  _{\hat{X}%
}+\left\|  f\right\|  _{L^{1}}.
\end{align*}
\end{proof}

Theorem \ref{teoB} raises the question of whether it is possible to prove
similar results for $p\neq\frac{n}{n-1}.$

\begin{proposition}
\label{Har}Suppose that the Sobolev-Poincar\'{e} inequality (\ref{poi}) holds
for some $1\leq p\leq\frac{n}{n-1},$ and let $X(\Omega)$ and $Y(\Omega)$ be
two r.i. spaces. We have

(i) if $X(\Omega)$ and $Y(\Omega)$ are such that
\begin{equation}
\left\|  \int_{t}^{1}u^{1-1/p}g(u)\frac{du}{u}\right\|  _{\hat{Y}}\leq
c\left\|  g\right\|  _{\hat{X}},\text{\ \ \ \ \ \ }\forall g\in\hat{X};
\label{acota}%
\end{equation}
then
\begin{equation}
\left\|  f-f_{\Omega}\right\|  _{Y(\Omega)}\preceq\left\|  \nabla f\right\|
_{X(\Omega)}. \label{taoca}%
\end{equation}
(ii) If $p\neq\frac{n}{n-1},$ then it is not necessarily true, in general,
that (\ref{taoca}) implies (\ref{acota}).
\end{proposition}

\begin{proof}
(i) The proof given in Theorem \ref{teoB} for $p=\frac{n}{n-1}$ works without
any changes in the general case.

(ii) Let $1<s<\frac{n}{n-1},$ and let $\Omega$ be an $s-$John domain. Then
$\Omega\in$ $\mathcal{J}_{(n-1)s/n}$ (cf. \cite{HK}) therefore the following
Sobolev-Poincar\'{e} inequality holds
\[
\left(  \int_{\Omega}\left|  f(x)-f_{\Omega}\right|  ^{\frac{n}{(n-1)s}%
}dx\right)  ^{(n-1)s/n}\preceq\int_{\Omega}\left|  \nabla f(x)\right|  dx.
\]
Let $t>1$ be such that $s>\frac{t-1}{n-1},$ and let $r=\frac{nt}%
{(n-1)s+(1-t)}.$ Note that $1<t<r.$ We will show that the validity of the
Sobolev-Poincar\'{e} inequality for $s-$John domains (cf. \cite{KM})
\[
\left\|  f-f_{\Omega}\right\|  _{r}\preceq\left\|  \nabla f\right\|  _{t},
\]
(this corresponds to the choice $Y=L^{r},X=L^{t}$ in Theorem \ref{teoB}) does
not imply that the Hardy operator $Hg(t)=\int_{t}^{1}u^{-(n-1)s/n}g(u)du$ \ is
a bounded operator, $H:L^{t}\rightarrow L^{r}.$ The boundedness of $H$ can be
reformulated as a weighted norm inequality for the operator $g\rightarrow
\int_{x}^{1}g(u)du,$ namely
\begin{equation}
\left\|  \int_{x}^{1}g(u)du\right\|  _{L^{r}}\leq c\left\|  g(x)x^{(n-1)s/n}%
\right\|  _{L^{t}}. \label{sabiamos}%
\end{equation}
It is well known that (\ref{sabiamos}) holds iff (cf. \cite[Theorem 3 page
44]{Ma})
\begin{equation}
\sup_{a>0}\left(  \int_{0}^{a}1\right)  ^{1/r}\left(  \int_{a}^{1}\left(
u^{(n-1)st/n}\right)  ^{\frac{-1}{t-1}}du\right)  ^{\frac{t-1}{t}}<\infty.
\label{dos}%
\end{equation}
Now, since $s<\frac{n}{n-1},$ it follows that $\frac{-(n-1)st}{n(t-1)}+1<0,$
and for $a$ near zero we have
\begin{align*}
\left(  \int_{0}^{a}1\right)  ^{1/r}\left(  \int_{a}^{1}\left(  s^{(n-1)st/n}%
\right)  ^{\frac{-1}{t-1}}\right)  ^{\frac{t-1}{t}}  &  \simeq a^{1/r}\left(
a^{\frac{-(n-1)st+n(t-1)}{n(t-1)}}-1\right)  ^{\frac{t-1}{t}}\\
&  \simeq a^{1/r}a^{\frac{-(n-1)st+n(t-1)}{nt}}\\
&  \simeq a^{\frac{(n-1)(1-t)(s-1)}{nt}}.
\end{align*}
Consequently, since $\frac{(n-1)(1-t)(s-1)}{nt}<0,$ (\ref{dos}) cannot hold.
\end{proof}

\begin{remark}
Let $h$ and $g$ be continuous, positive functions on an open set
$\Omega\subset\mathbb{R}^{n},$ and furthermore suppose that $\int_{\Omega
}h(x)dx<\infty$ (for simplicity we assume that $\int_{\Omega}h(x)dx=1).$ Let
$1<p<\infty,$ and assume that for every\footnote{If the weights are are
sufficiently nice the standard proof of density applies in order to extend
this inequality to Sobolev spaces.} $f\in C^{\infty}(\Omega),$ we
have\footnote{Several inequalities of the type (\ref{pesos1}) where $\Omega$
is a $s-$John domain ($s\geq1$), $h(x)=\varrho(x)^{a}$ and $g(x)=\varrho
(x)^{b}$ with $\varrho(x)=dist(x,\partial\Omega)$ can be found in \cite{KM}
and \cite{HK}.}
\begin{equation}
\left(  \int_{\Omega}\left|  f(x)-f_{\Omega,h}\right|  ^{p}h(x)dx\right)
^{\frac{1}{p}}\leq c\int_{\Omega}\left|  \nabla f(x)\right|  g(x)dx,
\label{pesos1}%
\end{equation}
(here $f_{\Omega,h}=\int_{\Omega}f(x)h(x)dx).$ Let $d\mu(x)=h(x)dx,$ then we
can rewrite (\ref{pesos1}) as
\begin{equation}
\left(  \int_{\Omega}\left|  f(x)-f_{\Omega,h}\right|  ^{p}d\mu(x)\right)
^{\frac{1}{p}}\leq c\int_{\Omega}\left|  \nabla f(x)\right|  \frac{g(x)}%
{h(x)}d\mu(x). \label{Poinpesos}%
\end{equation}
If we denote by $f_{\mu}^{\ast}$ the decreasing rearrangement of $f$ with
respect to the measure $\mu\ $and $f_{\mu}^{\ast\ast}(t)=\frac{1}{t}\int
_{0}^{t}f_{\mu}^{\ast}(s)ds$, then with the same proof of Theorem \ref{teoA},
we see that (\ref{Poinpesos}) and the following statements are equivalent:

\begin{enumerate}
\item [(i)]There exists $r_{f}\in\mathbb{R}$ such that
\[
s^{\frac{1}{p}-1}[\left(  f-r_{f}\right)  _{\mu}^{\ast\ast}(s)-\left(
f-r_{f}\right)  _{\mu}^{\ast}(s)]ds\preceq\int_{0}^{t}\left|  \nabla f\right|
_{\mu}^{\ast}(s)ds
\]
and
\[
\int_{0}^{t}s^{\frac{1}{p}-1}[\left(  f-r_{f}\right)  _{\mu}^{\ast\ast
}(s)-\left(  f-r_{f}\right)  _{\mu}^{\ast}(s)]ds\preceq\int_{0}^{t}\left|
\nabla f\right|  _{\mu}^{\ast}(s)ds.\
\]

\item[(ii)] For any rearrangement invariant space $X$
\[
\inf_{c}\left\|  s^{\frac{1}{p}-1}[\left(  f-c\right)  _{\mu}^{\ast\ast
}(s)-\left(  f-c\right)  _{\mu}^{\ast}(s)]\right\|  _{X}\preceq\left\|  \nabla
f\right\|  _{X}.
\]

\item[(iii)]
\[
\left\|  f-f_{\Omega,h}\right\|  _{L^{p,1}(\Omega,d\mu)}\preceq\left\|  \nabla
f\right\|  _{L^{1}(\Omega,d\mu)}.
\]
\end{enumerate}
\end{remark}

\section{\label{sec4}Symmetrization and Moduli of continuity}

In this brief section we formulate versions of the P\'{o}lya-Szeg\"{o}
principle for functions on domains.

Let $\Omega\in$ $\mathcal{J}_{1-1/n}$ be a domain of finite measure (for
simplicity we assume that $\left|  \Omega\right|  =1),$ and let $X(\Omega)$ be
a r.i. space. Given $f\in X(\Omega)$ the symmetric spherical decreasing
rearrangement $f^{\circ}$ of $f$\textbf{\ }is defined by
\[
f^{\circ}(x)=f^{\ast}\left(  \gamma_{n}\left|  x\right|  ^{n}\right)  ,\text{
\ \ }x\in B,
\]
where $\gamma_{n}=$ measure of the unit ball in $\mathbb{R}^{n}$ and $B$ is
the ball centered at the origin with $\left|  B\right|  =1.$ Since $f^{\circ}$
is equimeasurable with $f,$ $\left(  f^{\circ}\right)  ^{\ast}=f^{\ast}$,
$X(\Omega)$ has also a representation as a function space on $\tilde{X}(B)$
such that
\[
\left\|  f\right\|  _{X(\Omega)}=\left\|  f^{\circ}\right\|  _{\tilde{X}(B)}.
\]
Let us also recall that\footnote{We refer the reader to \cite{Mo} for further
information about symmetric spherical rearrangement.} (see \cite{MM1}).
\[
\left\|  f^{\circ}-g^{\circ}\right\|  _{\tilde{X}(B)}\leq\left\|  f-g\right\|
_{X}\text{ \ }f,g\in X.
\]

The first result of this section is an extension of the classical
P\'{o}lya-Szeg\"{o} inequality for domains of class $\mathcal{J}_{1-1/n}.$

\begin{theorem}
\label{Tpolya}Let $\Omega\in$ $\mathcal{J}_{1-1/n}$ and $X(\Omega)$ a r.i.
space. Then for any $f\in W^{1,1}(\Omega)$ we get that
\[
\inf_{c\in\mathbb{R}}\left\|  \nabla(f-c)^{\circ}\right\|  _{\tilde{X}%
(B)}\preceq\left\|  \nabla f\right\|  _{X(\Omega)}.
\]
\end{theorem}

\begin{proof}
We argue as in \cite{MMP}. Let $f\in W^{1,1}(\Omega),$ then a slight
modification to the proof of (\ref{pol01}) above yields that there is
$r_{f}\in\mathbb{R}$ such that for any Young function $\Phi$ we have
\[
\int_{0}^{1}\Phi\left(  s^{1/p}\left(  -(f-r_{f})^{\ast}\right)  ^{^{\prime}%
}(s)\right)  ds\preceq\int_{\Omega}\Phi(\left|  \nabla f(x)\right|  )dx.
\]
Since
\[
\int_{0}^{1}\Phi\left(  s^{1-1/n}\left(  -(f-r_{f})^{\ast}\right)  ^{^{\prime
}}(s)\right)  ds\simeq\int_{B}\Phi(\left|  \nabla(f-r_{f})^{\circ}(x)\right|
)dx
\]
we obtain,
\[
\int_{0}^{t}(\left|  \nabla(f-r_{f})^{\circ}\right|  )^{\ast}(s)ds\preceq
\int_{0}^{t}\left|  \nabla f\right|  ^{\ast}(s)ds.
\]
Summarizing, we get
\[
\inf_{c\in\mathbb{R}}\left\|  \nabla(f-c)^{\circ}\right\|  _{\tilde{X}(B)}%
\leq\left\|  \nabla(f-r_{f})^{\circ}\right\|  _{\tilde{X}(B)}\preceq\left\|
\nabla f\right\|  _{X(\Omega)}.
\]
\end{proof}

Given $f\in X\left(  {\Omega}\right)  ,$ the $X(\Omega)-$modulus of continuity
of $f$ is defined by
\begin{equation}
\omega_{X}(f,t)=\sup_{0<\left|  h\right|  \leq t}\left\|  (f(\cdot
+h)-f(\cdot))\chi_{\Omega(h)}\right\|  _{X(\Omega)},\label{modulo}%
\end{equation}
with $\Omega(h)=\left\{  x\in\Omega:x+\rho h\in\Omega,\text{ }0\leq\rho
\leq1\right\}  $, $h\in\mathbb{R}^{n}$.

Let $W_{X}^{1}=W_{X}^{1}(\Omega)=\{f\in X(\Omega):\nabla f\in X(\Omega)\}.$
Then, using the previous result, the fact that (cf. \cite{hs})
\[
\inf_{g\in W_{X}^{1}}\{\left\|  f-g\right\|  _{X(\Omega)}+t\left\|  \nabla
g\right\|  _{X(\Omega)}\}\simeq\omega_{X}(t,h),
\]
together with the proof of Theorem 1 in \cite{MM1}, we obtain

\begin{theorem}
\label{Tmodulo}Let $\Omega$ be an open domain in $\mathbb{R}^{n}$ with
Lipschitz boundary and such that $\left|  \Omega\right|  =1.$ Let $X\left(
{\Omega}\right)  $ a r.i. space, and let $f\in X\left(  {\Omega}\right)  .$
Then
\[
\inf_{c\in\mathbb{R}}\omega_{\tilde{X}(B)}(\left(  f-c\right)  ^{\circ
},t)\preceq\omega_{X(\Omega)}(f,t)\text{ \ \ \ }0<t<1.
\]
\end{theorem}

\begin{corollary}
Let $\Omega$ be an open domain in $\mathbb{R}^{n}$ with Lipschitz boundary and
such that $\left|  \Omega\right|  =1.$ Let $X\left(  {\Omega}\right)  $ a r.i.
space, and $f\in X\left(  {\Omega}\right)  .$ Then
\[
\inf_{c\in\mathbb{R}}\left(  \left(  f-c\right)  ^{\ast\ast}(t)-\left(
f-c\right)  ^{\ast}(t)\right)  \preceq\frac{\omega_{X(\Omega)}\left(
t^{1/n},f\right)  }{\phi_{X}(t)},\text{ \ \ \ }0<t<1/2,
\]
where $\phi_{X}(s)$ is the fundamental function\textbf{\ }of $X(\Omega):$
$\phi_{X}(s)=\left\|  \chi_{E}\right\|  _{X},$ with $E$ any measurable subset
of $\Omega$ with $\left|  E\right|  =s.$
\end{corollary}

\begin{proof}
By the previous theorem and since $\left(  \left(  f-c\right)  ^{\circ
}\right)  ^{\ast}=\left(  f-c\right)  ^{\ast}$ it is enough to check that for
any $c\in\mathbb{R}$%
\[
\left(  f-c\right)  ^{\ast\ast}(t)-\left(  f-c\right)  ^{\ast}(t)\preceq
\frac{\omega_{\tilde{X}(B)}(\left(  f-c\right)  ^{\circ},t^{1/n})}{\phi
_{X}(t)}\text{ \ \ \ }0<t<1/2,
\]
and this follows easily from Theorem 2 of \cite{MM1}.
\end{proof}

\section{Appendix}

In this section for completeness sake we provide a proof of Lemma \ref{01}. In
fact, the proof given below is implicitly contained in the proof of Theorem
1.2 of \cite{Ci}.

\begin{proof}
The main step is to show that for every finite family of intervals $\left(
a_{i},b_{i}\right)  ,$ $i=1,\ldots,m$, with $0<a_{1}<b_{1}\leq a_{2}<b_{2}%
\leq\cdots\leq a_{m}<b_{m}<\infty,$ there is a constant $c$ such that
\begin{equation}
\sum_{i=1}^{n}\int_{a_{i}}^{b_{i}}g(s)ds\leq c\int_{0}^{\sum_{i=1}^{n}%
(b_{i}-a_{i})}h^{\ast}(s)ds, \label{pas1}%
\end{equation}

If (\ref{pas1}) holds then by a routine limiting process we can show that for
any measurable set $E\subset$ $(0,\infty),$ we have
\[
\int_{E}g(s)\leq c\int_{0}^{|E|}h^{\ast}(s)ds,
\]
and the desired inequality follows:%
\[
\int_{0}^{t}g^{\ast}(s)\leq c\int_{0}^{t}h^{\ast}(s)ds,t>0.
\]

It remains to prove (\ref{pas1}). Fix $j$ $\in\left\{  1,\ldots,m\right\}  ,$
then by (\ref{h1})
\begin{align}
\sum_{i\geq j}\int_{a_{i}}^{b_{i}}g(s)  &  \leq c\int_{0}^{\infty}\chi
_{\cup_{i\geq j}(a_{i},b_{i})}(r)\left(  \frac{1}{r}\int_{0}^{r}h^{\ast
}(s)ds\right)  dr\label{utz}\\
&  =c\int_{0}^{\infty}h^{\ast}(s)\left(  \int_{s}^{\infty}\chi_{\cup_{i\geq
j}(a_{i},b_{i})}(r)\frac{dr}{r}\right)  ds.\nonumber
\end{align}
Since for $R\geq0$ we have (see \cite[formula (3.37) pag. 63 ]{Ci})
\[
\int_{0}^{R}\left(  \int_{s}^{\infty}\chi_{\cup_{i\geq j}(a_{i},b_{i}%
)}(r)\frac{dr}{r}\right)  ds\leq\left(  1+\sum_{i\geq j}\log\left(
\frac{b_{i}}{a_{i}}\right)  \right)  \int_{0}^{R}\chi_{\left[  0,\sum_{i\geq
j}(b_{i}-a_{i})\right]  }(s)ds
\]
by Hardy's Lemma (see \cite[Proposition 3.6 pag 63]{BS}) it follows that
\begin{align*}
&  \int_{0}^{\infty}h^{\ast}(s)\left(  \int_{s}^{\infty}\chi_{\cup_{i\geq
j}(a_{i},b_{i})}(r)\frac{dr}{r}\right)  ds\\
&  \leq\left(  1+\sum_{i\geq j}\log\left(  \frac{b_{i}}{a_{i}}\right)
\right)  \int_{0}^{\infty}h^{\ast}(s)\chi_{\left[  0,\sum_{i\geq j}%
(b_{i}-a_{i})\right]  }(s)ds,
\end{align*}
which combined with (\ref{utz}) gives
\begin{equation}
\sum_{i\geq j}\int_{a_{i}}^{b_{i}}g(s)ds\leq c\left(  1+\sum_{i\geq j}%
\log\left(  \frac{b_{i}}{a_{i}}\right)  \right)  \int_{0}^{\sum_{i\geq
j}(b_{i}-a_{i})}h^{\ast}(s)ds. \label{dd}%
\end{equation}
If $\sum_{i\geq j}\log\left(  \frac{b_{i}}{a_{i}}\right)  \leq1$, then
(\ref{pas1}) follows simply by choosing $j=1.$ If $\sum_{i=1}^{m}\log\left(
\frac{b_{i}}{a_{i}}\right)  >1,$ it is easily seen that there exist and index
$j_{0}$ and a positive number $c_{j_{0}\text{ }}$such that $a_{j_{0}\text{ }%
}\leq c_{j_{0}\text{ }}\leq b_{j_{0}\text{ }}$ and
\begin{equation}
1<\log\left(  \frac{b_{j_{0}\text{ }}}{c_{j_{0}\text{ }}}\right)
+\sum_{i>j_{0}}\log\left(  \frac{b_{i}}{a_{i}}\right)  \leq2. \label{dd1}%
\end{equation}
Applying (\ref{dd}) with $j=j_{0}$ replacing $a_{j_{0}\text{ }}$by
$c_{j_{0}\text{ }},$ we get
\begin{align*}
&  \int_{c_{j_{0}\text{ }}}^{b_{j_{0}}}g(s)ds+\sum_{i>j_{0}}\int_{a_{i}%
}^{b_{i}}g(s)ds\\
&  \leq c\left(  1+\log\left(  \frac{b_{j_{0}\text{ }}}{c_{j_{0}\text{ }}%
}\right)  +\sum_{i>j_{0}}\log\left(  \frac{b_{i}}{a_{i}}\right)  \right)
\int_{0}^{(b_{j_{0}\text{ }}-c_{j_{0}\text{ }})+\sum_{i>j_{0}}(b_{i}-a_{i}%
)}h^{\ast}(s)ds\\
&  \leq3c\int_{0}^{(b_{j_{0}\text{ }}-c_{j_{0}\text{ }})+\sum_{i>j_{0}}%
(b_{i}-a_{i})}h^{\ast}(s)ds\text{ \ \ (by (\ref{dd1})).}%
\end{align*}
On the other hand
\begin{align*}
\log\left(  \frac{b_{j_{0}\text{ }}}{c_{j_{0}\text{ }}}\right)  +\sum
_{i>j_{0}}\log\left(  \frac{b_{i}}{a_{i}}\right)   &  =\int_{c_{j_{0}}%
}^{\infty}\chi_{\left(  (c_{j_{0}},b_{j_{0}\text{ }})\cup\left(  \cup
_{i>j_{0}}(a_{i},b_{i})\right)  \right)  }(s)\frac{ds}{s}\\
&  \leq\frac{1}{c_{j_{0}\text{ }}}\left(  (b_{j_{0}\text{ }}-c_{j_{0}\text{ }%
})+\sum_{i>j_{0}}(b_{i}-a_{i})\right)  ,
\end{align*}
and since by (\ref{dd1})
\begin{equation}
c_{j_{0}\text{ }}<\left(  (b_{j_{0}\text{ }}-c_{j_{0}\text{ }})+\sum_{i>j_{0}%
}(b_{i}-a_{i})\right)  , \label{aa1}%
\end{equation}
it follows that
\begin{align*}
&  \sum_{1\leq j<j_{0}}\int_{a_{i}}^{b_{i}}g(s)ds+\int_{a_{j_{0}}}^{c_{j_{0}}%
}g(s)ds\\
&  \leq\int_{0}^{c_{j_{0}}}g(s)ds\leq c\int_{0}^{c_{j_{0}}}h^{\ast}(s)ds\text{
\ (by (\ref{h2}))}\\
&  \leq c\int_{0}^{(b_{j_{0}\text{ }}-c_{j_{0}\text{ }})+\sum_{i>j_{0}}%
(b_{i}-a_{i})}h^{\ast}(s)ds\text{ \ (by (\ref{aa1})). \ }%
\end{align*}
Summarizing
\begin{align*}
\sum_{i=1}^{n}\int_{a_{i}}^{b_{i}}g(s)ds  &  \leq\sum_{1\leq j<j_{0}}%
\int_{a_{i}}^{b_{i}}g(s)ds+\int_{a_{j_{0}}}^{c_{j_{0}}}g(s)ds+\int
_{c_{j_{0}\text{ }}}^{b_{j_{0}}}g(s)ds+\sum_{i>j_{0}}\int_{a_{i}}^{b_{i}%
}g(s)ds\\
&  \leq4c\int_{0}^{(b_{j_{0}\text{ }}-c_{j_{0}})+\sum_{i>j_{0}}(b_{i}-a_{i}%
)}h^{\ast}(s)ds\\
&  \leq4c\int_{0}^{\sum_{i=1}^{n}(b_{i}-a_{i})}h^{\ast}(s)ds.
\end{align*}
\end{proof}


\begin{thebibliography}{99}
\bibitem{BMR}J. Bastero, M. Milman and F. Ruiz, \textsl{A note on }%
$L(\infty,q)$\textsl{\ spaces and Sobolev embeddings}, Indiana Univ. Math. J.
\textbf{52} (2003), 1215-1230.

\bibitem{BS}C. Bennett and R. Sharpley, \textsl{Interpolation of Operators},
Academic Press, Boston\textbf{, }1988

\bibitem{BK}S. Buckley and P. Koskela, \textsl{Sobolev-Poincar\'{e} implies
John, }Math. Res. Lett. \textbf{2} (1995), 577--593.

\bibitem{Ci}A. Cianchi, \textsl{Symmetrization and second-order Sobolev
inequalities,} Annali di Matematica \textbf{183} (2004), 45-77.

\bibitem{Ha}P. Hajlasz,\textsl{\ Sobolev inequalities, truncation method, and
John domains}, Papers in Analysis, Rep. Univ. Jyv\"{a}skyl\"{a} Dep. Math.
Stat. 83, Univ. Jyv\"{a}skyl\"{a}, Jyv\"{a}skyl\"{a}, 2001, pp 109-126.

\bibitem{HK}P. Hajlasz and P. Koskela, \textsl{Isoperimetric inequalities and
Imbedding theorems in irregular domains,} J. London Math. Soc. \textbf{58}
(1998), 425--450.

\bibitem{hs}H. Johnen and K. Scherer, \textsl{On the equivalence of the }%
$K-$\textsl{functional and moduli of continuity and some applications}, in
Constructive theory of functions of several variables, pp. 119-140, Lecture
Notes in Math. 571, Springer, Berlin, 1977

\bibitem{KM}T. Kilpel\"{a}inen and J. Maly, \textsl{J. Sobolev inequalities on
sets with irregular boundaries,} Z. Anal. Anwendungen \textbf{19} (2000), 369--380.

\bibitem{KO}P. Koskela and P. Onninen, \textsl{Shap inequalities via
truncation,} J. Math. Anal. Appl. \textbf{278} (2003), 324--334.

\bibitem{MM}J. Mart\'{i}n and M. Milman, \textsl{Higher order symmetrization
inequalities and applications}, J. Math. Anal. Appl. \textbf{330} (2007), 91-113

\bibitem{MM1}J. Mart\'{i}n and M. Milman, \textsl{Symmetrization inequalites
and Sobolev embeddings}, Proc. Amer. Math. Soc. \textbf{134} (2006), 2335-2347.

\bibitem{MMP}J. Mart\'{i}n, M. Milman and E. Pustylnik, \textsl{Self Improving
Sobolev Inequalities, Truncation and Symmetrization,} to appear in J. Funct. Anal.

\bibitem{Ma}V. G. Maz'ya, \textsl{Sobolev Spaces, }Springer-Verlag, New York, 1985.

\bibitem{MP}M. Milman and E. Pustylnik, \textsl{On sharp higher order Sobolev
embeddings}, Comm. Contemp. Math. \textbf{6} (2004), 495-511.

\bibitem{Mo}Mossino, J, \textsl{In\'{e}galit\'{e}s isop\'{e}rimertriques et
applications on physique, }Trabaux en Cours, Hermann, Paris, 1984.

\bibitem{Ta}G. Talenti,\textsl{\ Inequalities in rearrangement-invariant
function spaces}, Nonlinear Analysis, Function Spaces and Applications,
Prometheus, Prague vol. 5, 1995, pp. 177-230.

\bibitem{Zi}W. Ziemer, \textsl{Weakly differentiable functions}, Graduate
Texts in Mathematics \textbf{120} (Springer, 1989).
\end{thebibliography}
\end{document}